\documentclass[11pt]{article}
\usepackage[english,activeacute]{babel}
\usepackage{amsmath,amsfonts,amssymb,amstext,amsthm,amscd,mathrsfs,amsbsy,centernot}
\usepackage{xcolor}
\usepackage{MnSymbol}
\usepackage{graphicx}
\usepackage[all]{xy}
\usepackage{hyperref}

\newtheorem{teo}{Theorem}[section]
\newtheorem{pro}[teo]{Proposition}
\newtheorem{coro}[teo]{Corollary}
\newtheorem{lem}[teo]{Lemma}

\theoremstyle{definition}
\newtheorem{defi}[teo]{Definition}
\newtheorem{exam}[teo]{Example}
\newtheorem{rem}[teo]{Remark}

\newcommand{\N}{\mathbb N}

\newcommand{\K}{\mathbb K}

\newcommand{\ug}{\underline{g}}
\newcommand{\vp}{\varphi}

\newcommand{\lt}{\ltimes}
\newcommand{\unD}{\underline{D_0}}

\newcommand{\hs}{\operatorname{HSDer}}

\newcommand{\hsk}{\operatorname{HSDer}_{k}^n}
\newcommand{\der}{\operatorname{Der}}

\newcommand{\hm}{\operatorname{Hom}}

\newcommand{\mysection}[2]{%
  \section[\texorpdfstring{#2}{#2}]{\texorpdfstring{#1}{#2}}%
}

\providecommand{\keywords}[1]{{\textbf{Keywords:}} #1}
\providecommand{\msc}[1]{{\textbf{MSC:}} #1}

\begin{document}

\title{$n$-trivial extensions and multi Hasse--Schmidt derivations}
\author{Paul Barajas\footnote{Supported by Universidad Nacional Aut\'onoma de M\'exico Postdoctoral Program (POSDOC).}, Daniel Duarte\footnote{Research supported by SECIHTI project CF-2023-G-33.}} 

\maketitle

\begin{abstract}
We propose a generalization of Hasse--Schmidt derivations that is equivalent to the notion of $n$-trivial extension introduced by Anderson-Bennis-Fahid-Shaiea, in the same way that derivations are equivalent to trivial extensions. We provide many examples of this generalization and prove some of its basic properties.
\end{abstract}

\noindent\keywords{$n$-trivial extension, Hasse--Schmidt derivations, jet space.}
\\
\msc{13N15, 13B99}

\section*{Introduction}

In commutative algebra and algebraic geometry, it is often useful to enrich a ring with additional structure to encode extra information. A classical example of this idea is provided by the dual numbers $\mathbb{K}[\varepsilon]/(\varepsilon^2)$, which can be seen as a first-order infinitesimal extension of the field $\mathbb{K}$.

There is a more general construction that extends this idea. Given a ring $A$ and an $A$-module $M$, the \emph{trivial extension} of $A$ by $M$ is defined as
\[A \ltimes M := A \oplus M,\]
with multiplication
\[(a,m) \cdot (a',m') = (aa',a m' + a' m).\]
In this framework, the dual numbers arise as the simplest case with $A = \mathbb{K}$ and $M = \mathbb{K}\varepsilon$, where $\varepsilon = (0,1)$.

Trivial extensions are intimately connected with derivations of $k$-algebras over modules. A key observation is that such derivations can be naturally interpreted as ring homomorphisms factoring through a trivial extension. More precisely, there is a canonical isomorphism
\begin{equation}\label{iso der tr-ext}
\der_k(A,M)\cong\{\phi\in\hm_{k-alg}(A,A\lt M)|\pi\circ\phi=Id_A\},
\end{equation}
where $\pi:A\lt M\to A$ is the projection to $A$.

The notion of trivial extension was generalized by Anderson-Bennis-Fahid-Shaiea by considering not only a single module $M$ but a finite family of $A$-modules $\{M_i\}_{i=1}^n$ \cite{ABFS}. Their generalization is called the $n$-trivial extension of $A$ by $\{M_i\}_{i=1}^n$. The authors conducted a very large study showing generalizations of known results on the usual trivial extension as well as providing new ones. 

In this paper we are interested in the following question: is there an algebraic object defined in terms of $A$ and $\{M_i\}_{i=1}^n$, extending derivations and satisfying an isomorphism as in (\ref{iso der tr-ext})? We answer this question when each $M_i$ is an $A$-algebra. Whenever $M_1=\cdots=M_n$, the answer is given by Hasse--Schmidt derivations. We propose a generalization of such derivations for different $A$-algebras $M_i$. We provide many examples and prove some basic properties of this generalization.

The original motivation for this study came from a formula established by de Fernex-Docampo that describes the module of K\"ahler differentials of the Hasse--Schmidt algebra \cite[Theorem 5.3]{dFDo1}. A key step in the proof of that theorem is precisely the correspondence (\ref{iso der tr-ext}). We were interested in finding a similar formula for the module of higher-order differentials, which motivated our search of higher-order versions of trivial extensions (see \cite[Remark 2.20]{BD} for details). 

All rings considered in this paper are commutative and have a unit element. Moreover, modules over a ring are considered with the natural structure of bimodule.

\mysection{$n$-trivial extensions}{n-trivial extensions}

In this section we present the definition of $n$-trivial extension introduced by Anderson, Bennis, Fahid, and Shaiea \cite{ABFS}. We constantly use the notation of the following theorem throughout the paper.

\begin{teo}\cite[Section 2]{ABFS}\label{a}
Let $A$ be a ring and $M_1,\ldots,M_n$ be $A$-modules, where $n\geq1$. Consider maps 
$$\varphi_{i,j}:M_i \times M_j\rightarrow M_{i+j},$$
for $1\leq i,j \leq n-1$ and $i+j\leq n$. Denote $m_im_j:=\varphi_{ij}(m_i,m_j)$ and $M_0=A$. Assume that:
\begin{itemize}
\item $\vp_{i,j}$ are $A$-bilinear for each $i,j$.
\item $m_im_j=m_jm_i$ for each $i,j$.
\item $(m_im_j)m_k=m_i(m_jm_k)$ for $0\leq i\leq j\leq k\leq n-1$.
\end{itemize}
Let $A\lt_n M_1\lt\cdots\lt M_n$ denote the direct sum $A\oplus M_1\oplus\cdots\oplus M_n$ with operations the usual addition and multiplication given by
$$(m_0,\ldots,m_n)(m_0',\ldots,m_n')=\Big(m_0m_0',m_0m_1'+m_1m_0',\ldots,\sum_{j+k=n}m_jm_k'\Big).$$
Then $A\lt_n M_1\lt\cdots\lt M_n$ is an associative and commutative ring with $(1,0,\ldots,0)$ as unit. In addition, $A\to A\lt_n M_1\lt\cdots\lt M_n$, $m_0\mapsto(m_0,0,\ldots,0)$ is a homomorphism of rings. By convention, $A\lt_0 M:=A$.
\end{teo}

\begin{defi}\cite{ABFS}
The $A$-algebra $A\lt_n M_1\lt\cdots\lt M_n$ is called the $n$-trivial extension of $A$ by $M_1,\ldots,M_n$. In the case $M_1=\cdots=M_n=:M$, the $n$-trivial extension is denoted $A\lt_n M$. Notice that the $1$-trivial extension $A\lt_1 M_1$ is the usual trivial extension of $A$ by $M_1$.
\end{defi}

Reference \cite{ABFS} provides numerous generalizations of known results on the usual trivial extension to the setting of $n$-trivial extensions, as well as several new results and questions.

\mysection{Hasse--Schmidt derivations and $n$-trivial extensions}{Hasse--Schmidt derivations and n-trivial extensions}

As mentioned in the introduction, letting $A$ be a $k$-algebra and $M$ an $A$-module, we have the following isomorphism ($\pi:A\lt M\to A$ is the projection to $A$):

\begin{equation}\label{ext-der}
\der_k(A,M)\cong\{\phi\in\hm_{k-alg}(A,A\lt M)|\pi\circ\phi=Id_A\}.\notag
\end{equation}

A natural question presents itself: is there a higher-order version of a derivation satisfying an analogous correspondence in the case of $n$-trivial extensions? 

The answer is known when the modules $M_i$ are all equal to a common $A$-algebra $B$: the Hasse--Schmidt derivations.
To establish this connection, we recall that definition and then show their natural correspondence with the $n$-trivial extension.

\begin{defi}\cite{V}
Let $n\in\N\cup\{\infty\}.$ Let $k$ be a ring. Let $A$, $B$ be $k$-algebras and let $f:k\to A$ be the structural morphism. A Hasse--Schmidt derivation of order $n$ from $A$ to $B$ over $k$ is a sequence $(D_0,\ldots,D_n) $ (or $(D_0,D_1,\ldots)$ if $n=\infty$), where $D_0:A\to B$ is a $k$-algebra homomorphism, $D_i:A\to B$ are $k$-linear maps, $D_i(f(c))=0$ for $i\in\{1\ldots,n\}$ and $c\in k$, and for all $x,y\in A$ and all $l\in\{0,\ldots, n\}$, the maps $D_l$ satisfy the following rule: 
\[D_l(xy)=\sum_{i+j=l}D_i(x)D_j(y).\]
We denote by $\hsk(A,B)$ the set of Hasse--Schmidt derivations of order $n$ from $A$ to $B$ over $k$.
\end{defi}

\begin{rem}
In general, $\hsk(A,B)$ does not have a module structure.
\end{rem}

The following result generalizes the correspondence (\ref{iso der tr-ext}) and provides the link between Hasse--Schmidt derivations and $n$-trivial extensions.

\begin{pro}\label{hs-hom}
Let $A$ and $B$ be $k$-algebras. Let $g:A\to B$ be a $k$-algebra homomorphism and give $B$ structure of $A$-module via this map. Let $\vp_{i,j}:B\times B\to B$ be the multiplication on $B$ for $1\leq i,j \leq n-1$ and $i+j\leq n$. Denote
\begin{align}
\hsk(A,B)^{\dagger}&:=\{(D_0,\ldots,D_n)\in\hsk(A,B)|D_0=g\},\notag\\
\hm_{k-alg}(A,A\lt_n B)^{\dagger}&:=\{\phi\in\hm_{k-alg}(A,A\lt_n B)|\pi\circ\phi=Id_A\},\notag
\end{align}
where $A\lt_n B$ is a $k$-algebra via $k\to A\to A\lt_n B$ and $\pi:A\lt_n B\to A$ is the projection to $A$. Then the following is a biyective correspondence:
\begin{align}
&\hsk(A,B)^{\dagger}\leftrightarrow \hm_{k-alg}(A,A\lt_n B)^{\dagger}\notag\\
&(g,D_1,\ldots,D_n)\mapsto \phi=(Id_A,D_1,\ldots,D_n)\notag\\
&(g,\phi_1,\ldots,\phi_n) \leftmapsto  \phi=(Id_A,\phi_1,\ldots,\phi_n).\notag
\end{align}
\end{pro}
\begin{proof}
Let $(g,D_1,\ldots,D_n)\in\hsk(A,B)^{\dagger}$ and $\phi=(Id_A,D_1,\ldots,D_n)$. We show that $\phi \in \hm_{k-alg}(A,A\lt_n B)^{\dagger}$. Indeed, $\phi$ is additive since the $D_i's$ are additive. Let $x,y\in A$. Denote $D_0:=g$. Then, by the multiplication defined on $A\lt_n B$ and the $A$-module structure on $B$ we obtain,
\begin{align}
\phi(xy)&=\Big(xy,D_1(xy),\ldots,D_n(xy)\Big)\notag\\
&=\Big(xy,D_0(x)D_1(y)+D_0(y)D_1(x),\ldots,\sum_{j+k=n}D_j(x)D_k(y)\Big)\notag\\
&=\phi(x)\phi(y).\notag
\end{align}
Thus, $\phi$ is a homomorphism of rings. It is also a homomorphism of $k$-algebras because of the $k$-algebra structure given on $A\lt_n B$ and using that $D_i(\lambda)=0$ for $\lambda\in k$ and $i\in\{1,\ldots,n\}$. 

That $(g,\phi_1,\ldots,\phi_n)$ is in $\hsk(A,B)^{\dagger}$ for $\phi\in\hm_{k-alg}(A,A\lt_n B)^{\dagger}$ is proved similarly.
\end{proof}

The previous proposition relate Hasse--Schmidt derivations with the $n$-trivial extension $A \lt_n B$. What about for $A\lt_n B_1 \lt \cdots \lt B_n$, where the $B_i's$ are different $A$-algebras? We answer this question in the following section.

\begin{rem}
The previous question will be addressed through a generalization of Hasse--Schmidt derivations. This approach is somewhat expected, given the multiplication defined in Theorem \ref{a}. On the other hand, there exist other generalizations of derivations beyond Hasse–Schmidt derivations (such as higher-order derivations and differential operators \cite{Gr,N}). To the best of our knowledge, however, there is no generalization of trivial extensions that satisfies an analogous property (\ref{iso der tr-ext}) for these higher-order versions of derivations.
\end{rem}


\mysection{Multi Hasse--Schmidt derivations and $n$-trivial extensions}{Multi Hasse--Schmidt derivations and n-trivial extensions}

In this section we introduce a notion that extends Hasse--Schmidt derivations to a context in which several rings are involved, together with some functions defined among them. We call it \textit{multi Hasse--Schmidt derivations}. 

We first show how this generalization connects with the concept of $n$-trivial extension, together with other basic facts. Moreover, we present several examples illustrating some special features of multi Hasse--Schmidt derivations.

\begin{defi}\label{mhs}
Let $A,B_1,\ldots,B_n$ be $k$-algebras, where $n\geq1$, $f:k \to A$ the structural morphism. Let $\vp_{i,j}:B_i\times B_j\to B_{i+j}$ be any maps for $1\leq i,j \leq n-1$ and $i+j\leq n$, and denote $b_ib_j:=\varphi_{ij}(b_i,b_j)$. A multi Hasse--Schmidt derivation from $A$ to $B_1,\ldots,B_n$ over $k$ of order $n$ with respect to $\vp:=\{\vp_{i,j}\}$ is a sequence $(\unD,D_1,\ldots,D_n)$ satisfying:
\begin{enumerate}
\item $\unD=(D_0^1,\ldots,D_0^n)$, where $D_0^i:A\to B_i$ are $k$-algebra homomorphisms.
\item $D_i:A\to B_i$ are additive for $i\in\{1,\ldots,n\}$.
\item $D_i(\lambda):=D_i(f(\lambda))=0$, for $i\in\{1,\ldots,n\}$ and $\lambda\in k$.
\item $D_i(xy)=D_0^i(x)D_i(y)+D_0^i(y)D_i(x)+\sum_{\substack{j+k=i,\\0<j,k<i}}D_j(x)D_k(y),$ for each $i\in\{1,\ldots,n\}$. We refer to these identities simply as the Leibniz rule.
\end{enumerate}
The set of multi Hasse--Schmidt derivations is denoted $\hsk(A,B_1,\ldots,B_n;\vp)$.
\end{defi}

Some basic remarks are in order.

\begin{rem}
Let $B$ be an $A$-algebra and $g:A\to B$ its structural morphism. Set $B_1=B_2=\cdots=B_n=:B$ and let the maps $\vp_{i,j}$ be the multiplication on $B$. Denote 
\begin{align*}
\hsk(A,B;\varphi)&:=\hsk(A,B,\ldots,B;\varphi),\\
\hsk(A,B;\varphi)^{\dagger}&:=\{(\unD,D_1,\ldots,D_n)\in\hsk(A,B;\varphi)| D_0^i=g,1\leq i \leq n\}. 
\end{align*}
There is a natural bijection\[\hsk(A,B;\varphi)^{\dagger}\leftrightarrow\{(D_0,D_1,\ldots,D_n)\in\hsk(A,B)\,|\,\,D_0=g\}.\]
\end{rem}

\begin{rem}\label{mHS HS1}
Let $(\unD,D_1,\ldots,D_n)\in\hsk(A,B_1,\ldots,B_n;\vp)$. It follows that $(D_0^1,D_1)\in\mbox{HSDer}^1_k(A,B_1)$.
\end{rem}

\begin{rem}\label{mHS-i}
Let $(\unD,D_1,\ldots,D_n)\in\hsk(A,B_1,\ldots,B_n;\vp)$. For each $1\leq i\leq n$ we have $((D_0^1,\ldots,D_0^i),D_1,\ldots,D_i)\in\hs^i_k(A,B_1,\ldots,B_i;\vp).$
\end{rem}

\subsection{Basic properties of multi Hasse--Schmidt derivations}

Let us show some basic properties of multi Hasse--Schmidt derivations. We first show that this notion provides a generalization of Proposition \ref{hs-hom}, regarding the relation among derivations and trivial extensions.

\begin{pro}\label{hsn-hom}
Let $A,B_1,\ldots,B_n$ be $k$-algebras. Let $g_i:A\to B_i$ be $k$-algebra homomorphisms and give $B_i$ structure of $A$-module via $g_i$. Denote $\underline{g}=(g_1,\ldots,g_n)$. Let $\vp_{i,j}:B_i\times B_j\to B_{i+j}$ be maps satisfying the conditions of Theorem \ref{a}. Then the following is a bijective correspondence:
\begin{align}
\hsk(&A,B_1,\ldots,B_n;\vp)^{\dagger}\leftrightarrow \hm_{k-alg}(A,A\lt_n B_1\lt\cdots\lt B_n)^{\dagger}\notag\\
&(\underline{g},D_1,\ldots,D_n)\mapsto \phi=(Id_A,D_1,\ldots,D_n)\notag\\
&(\underline{g},\phi_1,\ldots,\phi_n)\leftmapsto \phi=(Id_A,\phi_1,\ldots,\phi_n)\notag
\end{align}
where $\hsk(A,B_1,\ldots,B_n;\vp)^{\dagger}$ and $\hm_{k-alg}(A,A\lt_n B_1\lt\cdots\lt B_n)^{\dagger}$ are defined as in Proposition \ref{hs-hom}.
\end{pro}
\begin{proof}
The proof is similar to that of Proposition \ref{hs-hom}. Let $(\ug,D_1,\ldots,D_n)\in\hsk(A,B_1,\ldots,B_n;\vp)^{\dagger}$ and $\phi=(Id_A,D_1,\ldots,D_n)$. This map is additive since the $D_i's$ are additive. Let $x,y\in A$. Denote $D_0^i:=g_i$. Then, by the multiplication defined on $A\lt_n B_1\lt\cdots\lt B_n$ and the $A$-module structure on $B_i$ we obtain,
\begin{align}
\phi(xy)&=\Big(xy,D_1(xy),\ldots,D_n(xy)\Big)\notag\\
&=\Big(xy,D_0^1(x)D_1(y)+D_0^1(y)D_1(x),\ldots,\notag\\
&\hspace{0.7cm},\ldots,D_0^n(x)D_n(y)+D_0^n(y)D_n(x)+\sum_{\substack{j+k=n,\\0<j,k<n}}D_j(x)D_k(y)\Big)\notag\\
&=\phi(x)\phi(y).\notag
\end{align}
Thus, $\phi$ is a homomorphism of rings and also of $k$-algebras. Finally, that the map in the opposite direction is well defined is proved similarly.
\end{proof}

In the following proposition we observe some connections among multi Hasse--Schmidt derivations and the classical Hasse--Schmidt derivations.

\begin{pro}\label{comm diag}
Let $A,B_1,\ldots,B_n$ be $k$-algebras and $\phi_{ij}:B_i\to B_j$ be $k$-algebra homomorphisms, $1\leq i\leq j\leq n$, where $\phi_{jj}=Id_{B_j}$, such that $\phi_{ik}=\phi_{jk}\circ \phi_{ij}$, for $i<j<k$. Consider maps $\vp_{ij}:B_i\times B_j\to B_{i+j}$, $(b_i,b_j)\mapsto\phi_{i i+j}(b_i)\phi_{ji+j}(b_j)$, for $1\leq i,j \leq n$ and $i+j\leq n$. Denote 
$$\hs_k^n(A,B_1,\ldots,B_n;\vp)^*:=\big\{(\unD,D_1,\ldots,D_n)|D_0^j=\phi_{1j}\circ D_0^1,1\leq j\leq n\big\}.$$
Then, for each $j\in\{1,\ldots,n\}$, we have the following commutative diagram, 
$$\xymatrix{\hs_k^n(A,B_1) \ar@{->}[r]^{\phi_{1j}^*} \ar@{->}[dr]_{\alpha_1} &\hs_k^j(A,B_j)  \\ 
& \hs_k^n(A,B_1,\ldots,B_n;\vp)^*,\ar@{->}[u]_{\beta_j}}$$
where, denoting $\unD=(\phi_{11}\circ D_0,\phi_{12}\circ D_0,\ldots,\phi_{1n}\circ D_0)$,
\begin{align}
\phi_{1j}^*(D_0,\ldots,D_n)&=(\phi_{1j}\circ D_0,\phi_{1j}\circ D_1,\ldots,\phi_{1j}\circ D_j),\notag\\
\alpha_1(D_0,\ldots,D_n)&=(\unD,\phi_{11}\circ D_1,\phi_{12}\circ D_2,\ldots,\phi_{1n}\circ D_n),\notag\\
\beta_j(\underline{E_0},\ldots,E_n)&=(E_0^j,\phi_{1j}\circ E_1,\phi_{2j}\circ E_2,\ldots,\phi_{j-1j}\circ E_{j-1},\phi_{jj}\circ E_j).\notag
\end{align}
\end{pro}

Before giving the proof of Proposition \ref{comm diag}, let us illustrate the maps $\alpha_1$ and $\beta_j$ in two simple examples to get familiar with the notation.

\begin{exam}\label{ej0}
Let $A$, $B_1$ and $B_2$ be $k$-algebras and consider a $k$-algebra homomorphism $\phi_{12}:B_1\to B_2$.
Let 
\begin{align*}
\varphi_{1,1}:&B_1\times B_1\to B_2\\
&(b,b')\mapsto \phi_{12}(b)\phi_{12}(b').
\end{align*}
Let $(E_0,E_1,E_2)\in\hs^2_k(A,B_1)$. By definition, $\alpha_1(E_0,E_1,E_2)=(\underline{D_0},D_1,D_2)$, where
$$D_0^1=E_0,\,\, D_0^2=\phi_{12}\circ E_0,\,\,\underline{D_0}=(D_0^1,D_0^2),\,\,  D_1=E_1,\,\, \mbox{ and }\,\, D_2=\phi_{12}\circ E_2.$$ 
Let us verify in this simple example that $(\unD,D_1,D_2)\in\hs_k^2(A,B_1,B_2;\varphi)^*$. Indeed, $D_0^1$ and $D_0^2$ are $k$-algebra homomorphisms by definition, $D_1$ and $D_2$ are additive, and $D_i(\lambda)=0$ for $\lambda\in k$ and $i=1,2$.
Moreover, $D_1$ satisfies the Leibniz rule by Remark~\ref{mHS HS1}. Finally, for $x,y\in A$,
\begin{align*}
    D_2(xy)=&\phi_{12}(E_2(xy))\\
    =&\phi_{12}(E_0(x)E_2(y)+E_1(x)E_1(y)+E_2(x)E_0(y))\\
    =&\phi_{12}(E_0(x))\phi_{12}(E_2(y))+\phi_{12}(E_1(x))\phi_{12}(E_1(y))\\
    &+\phi_{12}(E_2(x))\phi_{12}(E_0(y))\\
    =&D_0^2(x)D_2(y)+D_1(x)D_1(y)+D_2(x)D_0^2(y).
\end{align*}

\end{exam}

\begin{exam}\label{ej4-0}
Let $A$, $B_1$, and $B_2$ be $k$-algebras. Let $\phi_{12}:B_1\to B_2$ be a $k$-algebra homomorphism and $\phi_{22}=Id_{B_2}$. Consider an element $(\underline{D_0},D_1,D_2)$ in $\hs^2_k(A,B_1,B_2;\varphi)^*$. In particular,
$D_0^2=\phi_{12}\circ D_0^1$.
Then
\[
\beta_2(\underline{D_0},D_1,D_2)=(D_0^2,\phi_{12}\circ D_1,\phi_{22}\circ D_2)=(D_0^2,\phi_{12}\circ D_1,D_2)
\in \hs^2_k(A,B_2).
\]
\end{exam}

Now we proceed to the proof of Proposition \ref{comm diag}.

\begin{proof}
The diagram commutes because of the hypothesis $\phi_{ik}=\phi_{jk}\circ \phi_{ij}$. Now we verify that these maps are well-defined. That $\phi_{1j}^*(D_0,\ldots,D_n)\in\hs_k^j(A,B_j)$ is a standard fact. Properties 1 to 3 of Definition \ref{mhs} follow for  $\alpha_1(D_0,\ldots,D_n)$ and $\beta_j(\underline{E_0},\ldots,E_n)$ from the corresponding properties of $(D_0,\ldots,D_n)$ and $(\underline{E_0},\ldots,E_n)$. Now we verify 4.

Let $a,a'\in A$ and $j\in\{1,\ldots,n\}$. Then,
\begin{align}
(\phi_{1j}\circ D_j)(aa')=&\phi_{1j}\big(D_0aD_ja'+D_0a'D_ja+\sum_{\substack{k+l=j,\\0<k,l<j}}D_kaD_la'\big)\notag\\
=&(\phi_{1j}\circ D_0)(a)(\phi_{1j}\circ D_j)(a')+(\phi_{1j}\circ D_0)(a')(\phi_{1j}\circ D_j)(a)\notag\\
&+\sum_{\substack{k+l=j,\\0<k,l<j}}\phi_{1k+l}(D_ka)\phi_{1k+l}(D_la')\notag\\
=&(\phi_{1j}\circ D_0)(a)(\phi_{1j}\circ D_j)(a')+(\phi_{1j}\circ D_0)(a')(\phi_{1j}\circ D_j)(a)\notag\\
&+\sum_{\substack{k+l=j,\\0<k,l<j}}\phi_{kk+l}(\phi_{1k}(D_ka))\phi_{lk+l}(\phi_{1l}(D_la'))\notag\\
=&(\phi_{1j}\circ D_0)(a)(\phi_{1j}\circ D_j)(a')+(\phi_{1j}\circ D_0)(a')(\phi_{1j}\circ D_j)(a)\notag\\
&+\sum_{\substack{k+l=j,\\0<k,l<j}}\vp_{k,l}(\phi_{1k}(D_ka),\phi_{1l}(D_la'))\notag\\
=&(\phi_{1j}\circ D_0)(a)(\phi_{1j}\circ D_j)(a')+(\phi_{1j}\circ D_0)(a')(\phi_{1j}\circ D_j)(a)\notag\\
&+\sum_{\substack{k+l=j,\\0<k,l<j}}(\phi_{1k}\circ D_k)(a)(\phi_{1l}\circ D_l)(a').\notag
\end{align}
Hence, $\alpha_1(D_0,\ldots,D_n)\in\hs_k^n(A,B_1,\ldots,B_n;\vp)^*$. 

Now fix $j\in\{1,\ldots,n\}$ and let $k\in\{1,\ldots,j\}$. Let $a,a'\in A$. Then,
\begin{align}
(\phi_{kj}\circ E_k)(aa')=&\phi_{kj}\big(E_0^kaE_ka'+E_0^ka'E_ka+\sum_{\substack{r+s=k,\\0<r,s<k}}E_raE_sa'\big)\notag\\
=&(\phi_{kj}\circ E_0^k)(a)(\phi_{kj}\circ E_k)(a')+(\phi_{kj}\circ E_0^k)(a')(\phi_{kj}\circ E_k)(a)\notag\\
&+\sum_{\substack{r+s=k,\\0<r,s<k}}\phi_{kj}\big(\vp_{r,s}(E_ra,E_sa')\big).\label{equa}
\end{align}
By definitions and hypothesis, $\phi_{kj}\circ E_0^k=\phi_{kj}\circ \phi_{1k}\circ E_0^1=\phi_{1j}\circ E_0^1=E_0^j$ and $\phi_{r+s,j}\big(\vp_{r,s}(E_ra,E_sa')\big)=(\phi_{rj}\circ E_r)(a)(\phi_{sj}\circ E_s)(a')$. Substituting these in (\ref{equa}) we obtain,
\begin{align}
(\phi_{kj}\circ E_k)(aa')=&E_0^j(a)(\phi_{kj}\circ E_k)(a')+E_0^j(a')(\phi_{kj}\circ E_k)(a)\notag\\
&+\sum_{\substack{r+s=k,\\0<r,s<k}}(\phi_{rj}\circ E_r)(a)(\phi_{sj}\circ E_s)(a').\notag
\end{align}
Hence, $\beta_j(\underline{E_0},\ldots,E_n)\in\hs_k^j(A,B_j)$.
\end{proof}

In the following subsection we provide examples showing that $\alpha_1$ is neither surjective nor injective in general, and that $\beta_j$ is not surjective in general.


\subsection{Examples}

We use the notation of Proposition \ref{comm diag} throughout this section. Our first example shows that $\alpha_1$ is not surjective in general.

\begin{exam}\label{ej2}
Let $\K$ be a field of characteristic zero. Let $A=B_1=\K[x]$, $B_2=\K[x,y]$, and $\phi_{12}:B_1\to B_2$ the inclusion. 
Define 
\begin{align*}
\varphi_{1,1}:&B_1\times B_1\to B_2\\
&(f,g)\mapsto \phi_{12}(f)\phi_{12}(g).
\end{align*}
Let $\unD=(Id_A,\phi_{12})$, $D_1=\partial_x$, $D_2=y\partial_x+\frac{1}{2}\partial_{x^2}$. 
A direct computation shows $(\unD,D_1,D_2)\in\hs_{\K}^2(A,B_1,B_2;\vp)^*$. 

Suppose that there exist $(E_0,E_1,E_2)\in\hs_{\K}^2(A,B_1)$ such that $$\alpha_1(E_0,E_1,E_2)=(\unD,D_1,D_2).$$

In particular, $D_2=\phi_{12}\circ E_2$. Evaluating at $x\in A$ we obtain
\[
y=D_2(x)=(\phi_{12}\circ E_2)(x)\in \phi_{12}(\K[x])=\K[x],
\]
a contradiction. Hence $\alpha_1$ is not surjective. 
\end{exam}

We now show that  $\alpha_1$ is not injective in general.

\begin{exam}\label{ej3}
Let $\K$ be a field of characteristic zero. Let $A=B_1=\K[x,y]$, $B_2=\K[x,y]/\langle y \rangle$, and $\phi_{12}:B_1\to B_2$ the projection. 
Define $\vp_{1,1}:B_1\times B_1\to B_2$, $\vp(f,g)=\phi_{12}(f)\phi_{12}(g)$. Let $\unD=(Id_A,\phi_{12})$, $D_1=y\partial_x$, $D_2=0$. 
A straightforward computation shows  $(\unD,D_1,D_2)\in\hs_{\K}^2(A,B_1,B_2;\vp)^*$. 

Consider the following two elements in $\hs^2_{\K}(A,B_1)$:
\[
E=(Id_A,\,y\partial_x,\,\tfrac{1}{2}y^2\partial_{x^2})
\qquad\text{and}\qquad
E'=(Id_A,\,y\partial_x,\,y\partial_x+\tfrac{1}{2}y^2\partial_{x^2}).
\]
Observe that $\phi_{12}(y)=0$. Hence $\phi_{12}\circ (y\partial_x)=0$ and
$\phi_{12}\circ (\tfrac{1}{2}y^2\partial_{x^2})=0$. Hence, $\phi_{12}\circ E_2=\phi_{12}\circ E_2'=0$.
Therefore
\[
\alpha_1(E)=(\unD,D_1,D_2)=\alpha_1(E').
\]
In particular, $\alpha_1$ is not injective.

\end{exam}

Next, an example showing that $\beta_j$ is not surjective in general.

\begin{exam}\label{ej4}
Let $\K$ be a field of characteristic zero. Let $A=B_1=\K[x]$, $B_2=\K[x,y]$, $\phi_{12}:B_1\to B_2$ the inclusion. Recall that $\phi_{22}=Id_{B_2}$. 
Consider $(\phi_{12},y\partial_x,\tfrac{1}{2}y^2\partial_{x^2})\in\hs_k^2(A,B_2)$. 
Suppose there exists $(\unD,D_1,D_2)\in\hs_k^2(A,B_1,B_2;\vp)^*$ such that 
$$\beta_2(\underline{D_0},D_1,D_2)=(D_0^2,\phi_{12}\circ D_1,\phi_{22}\circ D_2)=(\phi_{12},y\partial_x,\frac{y^2}{2}\partial_{x^2}).$$ 

Then $y\partial_x=\phi_{12}\circ D_1$, which implies $y\in\K[x]$, a contradiction.
\end{exam}

We conclude this section with two further examples of multi Hasse--Schmidt derivations.

\begin{exam}\label{ej5}
Let $A$, $B_1$ and $B_2=B_1 \times B_1$ be $k$-algebras and $\theta:A\to B_1$ a $k$-algebra homomorphism. 
Consider the map 
$$\vp_{1,1}:B_1\times B_1\to B_2,\,\,\,(b,b')\mapsto(0,bb').$$ 
Given $(\theta,D_1,D_2)\in\hs_k^2(A,B_1)$, define
\begin{align*}
&E_0^1:A\to B_1,\,\,\,a\mapsto \theta(a),\\
&E_0^2:A\to B_2,\,\,\,a\mapsto (\theta(a),\theta(a)),\\ 
&E_1:A\to B_1,\,\,\,a\mapsto D_1(a),\\
&E_2:A\to B_2,\,\,\,a\mapsto (0,D_2(a)). 
\end{align*}
Then $(\underline{E_0},E_1,E_2)\in\hs_k^2(A,B_1,B_2;\vp)$. 
\end{exam}

\begin{exam}\label{ej6}
Let $A$ be a $k$-algebra and $B_1,B_2,B_3$ be $A$-algebras. 
For each $i,j\in\{1,2,3\}$ with $i<j$, let $\phi_{ij}:B_i\to B_j$ be $k$-algebra homomorphisms such that $\phi_{13}=\phi_{23}\circ \phi_{12}.$ 
Let $x\in B_2$, $y_1\in B_3$, and $y_2:=\phi_{23}(x)\in B_3$. 
Define 
\begin{align*}
\vp_{1,1}&:B_1\times B_1\to B_2,\,\,\,(b,b')\mapsto x\phi_{12}(bb'), \\
\vp_{2,1}=\vp_{1,2}&:B_1\times B_2\to B_3,\,\,\,(b,b')\mapsto y_1y_2\phi_{13}(b)\phi_{23}(b').
\end{align*} 
Let $(D_0,D_1,D_2,D_3)\in\hs_k^3(A,B_1)$. 
Let $D_0^1=D_0$ and $D_0^i=\phi_{1i}\circ D_0$ for $i=2,3$. 
Set $\underline{D_0}=(D_0^1,D_0^2,D_0^3)$, and
\begin{align*}
&E_1=D_1,\\
&E_2=x(\phi_{12}\circ D_2)\\
&E_3=y_1y_2^2(\phi_{13}\circ D_3).
\end{align*}
Then $E=(\underline{D_0},E_1,E_2,E_3)\in\hs_k(A,B_1,B_2,B_3;\vp)$.
\end{exam}


\mysection{Multi Hasse--Schmidt derivations and $n$-jets}{Multi Hasse--Schmidt derivations and n-jets}

One of the essential features of Hasse--Schmidt derivations is the following bijective correspondence:
\begin{align}
\hsk&(A,B)\leftrightarrow \hm_{k-alg}(A,B[t]/\langle t^{n+1} \rangle)\notag\\
(D_0,\ldots&,D_n)\mapsto \phi(a)=\sum_{i=0}^n D_i(a)t^i.\notag
\end{align}
In this section we explore a similar correspondence for the multi Hasse--Schmidt derivations in the special case $A=B_1=\cdots=B_n$.

We consider maps 
\begin{align*}
\vp_{i,j}:&A\times A\to A\\    
&(a,a')\mapsto\lambda_{i,j}aa',
\end{align*}
where $\lambda_{i,j}\in A$, $\lambda_{i,j}=\lambda_{j,i}$, and $1\leq i,j \leq n-1$, $i+j\leq n$.

\begin{rem}
If $\lambda_{ij}=1$ for all $i,j$ then $A\lt_n A\cong A[t]/\langle t^{n+1} \rangle$, as $A$-algebras. The isomorphism is given by
$$A\lt_n A\to A[t]/\langle t^{n+1}\rangle,\,\,\,\,\,(a_0,\ldots,a_n)\mapsto \sum_{i=0}^na_it^i.$$
\end{rem}

Our goal is to describe $A\lt_n A$ with respect to the functions $\vp_{i,j}$. First some examples.

\begin{itemize}
\item If $n=1$ there are no $\vp_{ij}$ and $A\lt_1 A\cong A[t_1]/J_1$, where $J_1=\langle t_1^2 \rangle$. 
\item Let $n=2$ and $\vp_{1,1}(a,a')=\lambda_{1,1}aa'$. It is proved in \cite[Example 2.2]{ABFS} that
\begin{align}
&A\lt_2 A\cong A[t_1,t_2]/J_2\notag\\
(a_0&,a_1,a_2)\mapsto a_0+a_1t_1+a_2t_2,\notag
\end{align}
where $J_2=\langle t_1^2-\lambda_{1,1}t_2,t_1t_2,t_2^2 \rangle$. 
\item For $n=3$ and setting $\vp_{1,1}(a,a')=\lambda_{1,1}aa'$, $\vp_{1,2}(a,a')=\lambda_{1,2}aa'=\vp_{2,1}(a,a')$,  similar computations show that
\begin{align}
&A\lt_3 A\cong A[t_1,t_2,t_3]/J_3\notag\\
(a_0,a_1,&a_2,a_3)\mapsto a_0+a_1t_1+a_2t_2+a_3t_3,\notag
\end{align}
where $J_3=\langle t_1^2-\lambda_{1,1}t_2,t_1t_2-\lambda_{1,2}t_3,t_2^2,t_1t_3,t_2t_3,t_3^2\rangle$.
\end{itemize}

These examples are particular cases of the following theorem.

\begin{teo}\label{semijets}
Let $A$ be a $k$-algebra and consider maps $\vp_{ij}:A\times A\to A$, $(a,a')\mapsto\lambda_{i,j}aa'$, for some $\lambda_{i,j}\in A,$ where $1\leq i,j \leq n-1$, $i+j\leq n$, and such that $\vp_{i,j}(a,a')=\lambda_{i,j}aa'=\vp_{j,i}(a,a')$. Let $J_n\subset A[t_1,\ldots, t_n]$ be the ideal generated by the following sets,
\begin{align*}
I_1&=\{t_it_j-\lambda_{i,j}t_{i+j} \,|\, 1\leq i,j \leq n-1,\, i+j \leq n \},\\ 
I_2&=\{t_it_j \,|\,1\leq i,j \leq n,\,  i+j>n\}.
\end{align*} 
Then $A\lt_nA\simeq A[t_1,\ldots,t_n]/J_n$, as $A$-algebras.
\end{teo}

\begin{proof}
Let $\{e_0,\ldots,e_n\}$ denote the canonical basis of the module $A\lt_n A$. Let $t_0=1,$ $\lambda_{0,0}=\lambda_{0,j}=\lambda_{i,0}=1$. Consider the $A$-algebra homomorphism
\begin{align*} 
\overline{\psi}:A[t_1,\ldots,t_n] &\to A\lt_nA\\ 
1&\mapsto e_0,\\ 
t_i&\mapsto e_i.
\end{align*} 
Let $t_it_j-\lambda_{i,j}t_{i+j}\in I_1$ and $t_it_j\in I_2.$ We note that 
\begin{align*}
\overline{\psi}(t_it_j-\lambda_{i,j}t_{i+j})&=\overline{\psi}(t_i)\overline{\psi}(t_j)-\lambda_{i,j}\overline{\psi}(t_{i+j})=e_ie_j-\lambda_{i,j}e_{i+j}=0,\\
\overline{\psi}(t_it_j)&=\overline{\psi}(t_i)\overline{\psi}(t_j)=0.
\end{align*} 
Thus, $J_n=\langle I_1,I_2\rangle\subset\ker(\overline{\psi})$. Then $\overline{\psi}$ induces $\psi:A[t_1\ldots,t_n]/J_n\to A\lt_nA.$

Now we define a homomorphism in the other direction. Consider
\begin{align*} 
\phi:A\lt_nA&\to A[t_1,\ldots,t_n]/J_n\\
(a_0,\ldots,&a_n)\mapsto \sum_{i=0}^na_it_i.
\end{align*} 
Let us show that $\phi$ is a ring homomorphism. Let $(a_0,\ldots,a_n)$, $(b_0,\ldots,b_n)\in A\lt_nA.$ Recall that
$$(a_0,\ldots,a_n)\cdot(b_0,\ldots,b_n)=(a_0b_0,a_0b_1+a_1b_0,\ldots,\sum_{i+j=n}a_ib_j).$$
Since, by definition, $a_ib_j=\lambda_{i,j}a_ib_j,$ we have that:
\begin{align*} 
\phi\big((a_0,\ldots,a_n),(b_0,\ldots,b_n)\big)&=\phi\big( (a_0b_0,a_0b_1+a_1b_0,\ldots,\sum_{i+j=n}a_ib_j)\big)\\
&=\phi\big( (a_0b_0,a_0b_1+a_1b_0,\ldots,\sum_{i+j=n}\lambda_{i,j}a_ib_j)\big)\\
&=\sum_{k=0}^n\sum_{i+j=k}\lambda_{i,j}a_ib_jt_k.
\end{align*}
On the other hand,
\begin{align*}
\phi(a_0,\ldots,a_n)\phi(b_0,\ldots,b_n)&=\big(\sum_{i=0}^na_it_i\big)\big(\sum_{i=0}^nb_it_i\big)\\
&=\sum_{k=0}^{2n}\sum_{i+j=k}a_ib_jt_it_j\\
&=\sum_{k=0}^n\sum_{i+j=k}a_ib_jt_it_j,\\
&=\sum_{k=0}^n\sum_{i+j=k}a_ib_j\lambda_{i,j}t_{i+j}. 
\end{align*}

In order to prove that $\psi$ and $\phi$ are inverse maps, let us describe explicitly the elements of $A[t_1,\ldots,t_n]/J_n.$ 

 Let $t_i^{\alpha_i}\in A[t_1,\ldots,t_n]/J_n$ and assume that $2\leq\alpha_i\leq n.$ Suppose that $\alpha_ii>n$ and set $j=\max\{k\in\{1,\ldots,\alpha_i-1\}|ki\leq n\}$. Thus,
$$t_i^{\alpha_i}=t_i^2t_i^{\alpha_i-2}=\lambda_{i,i}t_{2i}t_i^{\alpha_i-2}=\lambda_{i,i}t_{2i}t_it_i^{\alpha_i-3}
=\cdots=\prod_{l=1}^j\lambda_{li,i} t_{ji}t_it_i^{\alpha_i-(j+1)}.$$ Since $ji,i\leq n$ and $ij+i>n$ then $t_{ij}t_i\in J_n$, i.e., $t_i^{\alpha_i}=0$. Now, suppose that $\alpha_ii\leq n$, then $t_{\alpha_ii}\in \{t_1,\ldots,t_n\}$. Hence $t_i^{\alpha_i}=\prod_{l=1}^{\alpha_i}\lambda_{li,i}t_{\alpha_ii}.$ Finally suppose that $\alpha_i>n$, since $\alpha_ii>n$ , then $t_i^{\alpha_i}=0.$

Now let $t^{\alpha}=t_1^{ \alpha_1}\cdots t_n^{\alpha_n}\in A[t_1,\ldots,t_n]/J_n$. Notice that if $\alpha_ii>n$ for some $i\in\{1,\ldots,n\}$, then $t^{\alpha}=t_1^{ \alpha_1}\cdots t_n^{\alpha_n}=0$. Assume that $\alpha_ii\leq n$ for all $i$. If $\sum_{i=1}^ni\alpha_i>n$, arguing as before, we obtain $t^{\alpha}=0$. Otherwise,
\begin{align*}
t_1^{ \alpha_1}\cdots t_n^{\alpha_n}=&\prod_{i=1}^n\lambda_it_{i\alpha_i}= \Big(\prod_{i=1}^n\lambda_i\Big)\lambda_{\alpha_1,2\alpha_2}t_{\alpha_1+2\alpha_2}t_{3\alpha_3}\cdots t_{n\alpha_n}\\
=&\Big(\prod_{i=1}^n\lambda_i\Big)\lambda_{\alpha_1,2\alpha_2}\lambda_{\alpha_1+2\alpha_2,3\alpha_3}t_{\alpha_1+2\alpha_2+3\alpha_3}t_{4\alpha_4}\cdots t_{n\alpha_n}\\
&\vdots\\
=&(\prod_{i=1}^n\lambda_i)\lambda_{\alpha_1,2\alpha_2}\lambda_{\alpha_1+2\alpha_2,3\alpha_3}\lambda_{\alpha_1+2\alpha_2+3\alpha_3,4\alpha_4}\cdots\lambda_{\sum_{i=1}^{n-1}i\alpha_i,n\alpha_n}t_{\sum_{i=1}^ni\alpha_i}.
\end{align*} 
Therefore every $f\in A[t_1,\ldots,t_n]/J_n$ can be written as $f=\sum_{i=0}^na_it_i.$ 
Using this fact, a straightforward computation shows that $\psi$ and $\phi$ are mutually inverse.
\end{proof}

\begin{coro}
With the previous notation, there is a bijective correspondence among
$$\{(\unD,\ldots,D_n)\in \hsk(A,A;\vp)|\unD=\underline{Id_A}\}$$
and 
$$\{\phi\in \hm_{k-alg}(A,A[t_1,\ldots,t_n]/J_n)|\phi\equiv Id_A\mod \langle t_1,\ldots,t_n\rangle\}.$$
The bijection is given by
$$(\unD,D_1,\ldots,D_n)\mapsto \phi(a)=a+\sum_{i=1}^n D_i(a)t_i.$$
\end{coro}
\begin{proof}
This follows from Theorem \ref{semijets} and Proposition \ref{hsn-hom}.
\end{proof}

\section{A functorial property}

Another classic result on Hasse--Schmidt derivations is the fact that it defines a representable functor. In this final section we discuss an elementary property of multi Hasse--Schmidt derivations pointing towards that direction.

Let $A$ be a $k$-algebra. Let $B_1,\ldots,B_n$ and $B_1',\ldots,B_n'$ be $k$-algebras and consider maps $\vp_{j,k}:B_j\times B_k\to B_{j+k}$, $\vp'_{j,k}:B_j'\times B_k'\to B_{j+k}'$ satisfying the conditions appearing in Theorem \ref{a}. 

\begin{lem}\label{vpstar}
With the previous notation, let $\psi_i:B_i\to B_i'$ be $k$-algebra homomorphisms for $i\in\{1,\ldots,n\}$ such that for all $1\leq j,k \leq n-1$ and $j+k\leq n$, the following diagram commutes
$$\xymatrix{B_j\times B_k \ar@{->}[r]^{\vp_{j,k}} \ar@{->}[d]_{\psi_j\times\psi_k} &B_{j+k} \ar@{->}[d]^{\psi_{j+k}}  \\  
B_j'\times B_k' \ar@{->}[r]^{\vp_{j,k}'} & B_{j+k}'.}$$
Then the following map is well defined:
\begin{align}
\psi^*:\hsk&(A,B_1,\ldots,B_n;\vp) \to \hsk(A,B_1',\ldots,B_n';\vp')\notag\\
&(\unD,D_1,\ldots,D_n)\mapsto(\psi\circ\unD,\psi_1\circ D_1,\ldots,\psi_n\circ D_n),\notag
\end{align}
where $\psi\circ\unD=(\psi_1\circ D_0^1,\ldots,\psi_n\circ D_0^n)$.
\end{lem}
\begin{proof}
Firstly, for each $i\in\{1,\ldots,n\}$, $\psi_i\circ D_0^i$ is a $k$-algebra homomorphism and $\psi_i\circ D_i$ is additive. Also, $\psi_i\circ D_i(\lambda)=\psi_i(0)=0$ for $\lambda\in k$. Finally, for each $i\in\{1,\ldots,n\}$ 
\begin{align}
\psi_i\circ D_i(xy)&=\psi_i(D_0^i(x)D_i(y)+D_0^i(y)D_i(x)+\sum_{\substack{j+k=i,\\0<j,k<i}}D_j(x)D_k(y))\notag\\
&=\psi_i(D_0^i(x))\psi_i(D_i(y))+\psi_i(D_0^i(y))\psi_i(D_i(x))\notag\\
&\hspace{0.8cm}+\sum_{\substack{j+k=i,\\0<j,k<i}}\psi_i(\vp_{j,k}(D_j(x),D_k(y)))\notag\\
&=\psi_i(D_0^i(x))\psi_i(D_i(y))+\psi_i(D_0^i(y))\psi_i(D_i(x))\notag\\
&\hspace{0.8cm}+\sum_{\substack{j+k=i,\\0<j,k<i}}\vp'_{j,k}(\psi_j(D_j(x),\psi_k(D_k(y)))\notag\\
&=\psi_i(D_0^i(x))\psi_i(D_i(y))+\psi_i(D_0^i(y))\psi_i(D_i(x))\notag\\
&\hspace{0.8cm}+\sum_{\substack{j+k=i,\\0<j,k<i}}\psi_j(D_j(x))\psi_k(D_k(y)).\notag
\end{align}
\end{proof}

Several natural questions arise from the previous lemma: \textit{does multi Hasse--Schmidt derivations define a functor?, among which categories?, is it representable?} Also, and importantly: \textit{is there a geometric interpretation of multi Hasse--Schmidt derivations?} We believe these questions deserve a deeper study.

\vspace{.5cm}
\noindent{\footnotesize \textsc {Paul Barajas, Universidad Nacional Aut\'onoma de M\'exico} \\
paul.barajas@im.unam.mx}\\
{\footnotesize \textsc {Daniel Duarte, Universidad Nacional Aut\'onoma de M\'exico} \\ adduarte@matmor.unam.mx}\\

\begin{thebibliography}{XXX}
\addcontentsline{toc}{section}{\numberline{References}}
\bibitem{ABFS}Anderson, D. D., Bennis, D., Fahid, B., Shaiea, A.; \textit{On $n$-trivial extensions of rings}, Rocky Mountain J. Math., Vol. 47, No. 8, (2017), pp 2439-2511.
\bibitem{BD} Barajas, P., Duarte, D.; \textit{A Nobile-like theorem for jet-schemes of hypersurfaces}, Osaka Journal of Math., Vol. 60, (2023), pp 555-569.
\bibitem{dFDo1}de Fernex, T., Docampo, R.; \textit{Differentials on the arc space}, Duke Math. J., Vol. 169, No. 2, (2020), pp 353-396.
\bibitem{Gr} Grothendieck, A.; \textit{\'El\'ements de g\'eom\'etrie alg\'ebrique. IV. \'Etude locale des sch\'emas et des
morphismes de sch\'emas IV}, Inst. Hautes \'Etudes Sci. Publ. Math., (32):361, 1967.
\bibitem{N}Nakai, Y.; \textit{High order derivations I}, Osaka J. Math., Vol. 7, (1970), pp 1-27.
\bibitem{V}Vojta, P.; \textit{Jets via Hasse-Schmidt derivations}, in Diophantine geometry, CRM Series, Vol. 4, Edizioni della Normale, (2007), pp 335-361.
\end{thebibliography}
\end{document}